\documentclass[12pt,oneside]{article}
\usepackage[inner=32mm,outer=31mm,tmargin=30mm,bmargin=40mm]{geometry}
\usepackage[utf8]{inputenc}
\usepackage[T1]{fontenc}
\usepackage{sectsty}
    \sectionfont{\large}
    \subsectionfont{\normalsize}
\usepackage{amssymb}
\usepackage{geometry}
\usepackage{amsmath}
\usepackage{amsthm}
\usepackage{graphicx}
\usepackage{color}
\usepackage[normalem]{ulem}
\usepackage[english]{babel}
\usepackage[onehalfspacing]{setspace}
\usepackage{verbatim}
\usepackage[hidelinks]{hyperref}

\newcommand{\IR}{\ensuremath{\mathbb{R}}}
\newcommand{\IN}{\ensuremath{\mathbb{N}}}

\newcommand{\IC}{\ensuremath{\mathbb{C}}}
\newcommand{\IK}{\ensuremath{\mathbb{K}}}
\newcommand{\IP}{\ensuremath{\mathbb{P}}}
\newcommand{\IE}{\ensuremath{\mathbb{E}}}

\newcommand{\norm}[1]{\left\Vert#1\right\Vert}

\newcommand{\set}[1]{\left\{#1\right\}}
\newcommand{\abs}[1]{\left|#1\right|}
\newcommand{\brackets}[1]{\left(#1\right)}
\renewcommand{\d}{{\rm d}} 
\newcommand{\scalar}[2]{\left\langle#1,#2\right\rangle}

\newcommand{\diff}{D}

\newtheorem{thm}{Theorem}
\newtheorem{cor}{Corollary}
\theoremstyle{plain}
\newtheorem{lemma}{Lemma}

\theoremstyle{definition}
\newtheorem{ex}{Example}

\newtheorem{rem}{Remark}
\newtheorem*{ack}{Acknowledgements}
\newtheorem*{alg}{Algorithm}

\title{Optimal Monte Carlo methods for $L^2$-approximation} 

\author{
David Krieg
\\ 
Mathematisches Institut, Universit\"at Jena\\ 
Ernst-Abbe-Platz 2, 07743 Jena, Germany  \\ 
david.krieg@uni-jena.de}

\date{\today}

\begin{document}

\maketitle

\begin{abstract}
 We construct Monte Carlo methods
 for the $L^2$-approximation in Hilbert spaces of multivariate functions
 sampling not more than $n$ function values of the target function.
 Their errors catch up with the rate of convergence and the preasymptotic behavior 
 of the error of any algorithm
 sampling $n$ pieces of arbitrary linear information,
 including function values.
\end{abstract}



\section{Introduction}

Assume we want to approximate an unknown real or complex valued function
on a set $D$ based on a finite number $n$ of function values
which may be evaluated at randomly and adaptively chosen points.
In general, these function values do not determine the function uniquely and so
we cannot expect our approximation to be correct.
We make an approximation error which we measure in the space $L^2(D,\mathcal{A},\mu)$
of quadratically integrable functions on $D$ with respect to an arbitrary measure $\mu$.
In order to make any meaningful statement regarding this error,
we need to have additional a priori knowledge of the unknown function.
Here, we assume structural knowledge 
of the form that it is contained in the unit ball $F_\circ$ of a Hilbert space
$F$ which is compactly embedded in $L^2(D,\mathcal{A},\mu)$.
For instance, it may be bounded with respect to some Sobolev norm on
a compact manifold $D$.
The error of the randomized algorithm or Monte Carlo method
$A_n$ is the quantity
\begin{equation*}
 e^{\rm ran}(A_n)
 = \sup\limits_{f\in F_\circ} \brackets{\IE \int_D \abs{f-A_n(f)}^2 ~\d\mu}^{1/2} 
.\end{equation*}
The error of an optimal randomized algorithm that ask for at most $n$
function values is denoted by
\begin{equation*}
 e(n)=\inf_{A_n} e^{\rm ran}(A_n).
\end{equation*}
While it seems impossible to provide such algorithms,
the optimal deterministic algorithm
evaluating $n$ arbitrary linear functionals
is well known.
It is given by the orthogonal projection $P_n$
onto the span of the first $n$ functions
in the singular value decomposition of the
embedding $T:F\hookrightarrow L^2$.
Its worst case error is the $(n+1)$-st largest singular value or approximation number
$\sigma(n+1)$ of that embedding, 
the square root of the $(n+1)$-st largest eigenvalue of 
the operator $W=T^*T$.

The algorithm $P_n$ asks for the first $n$ coefficients of $f$
with respect to the singular value decomposition of the embedding $T$.
In most applications, however, it is not possible to sample these coefficients
and we may only make use of function values.
This leads to the following questions:
\begin{itemize}
 \item How does the error $e(n)$ of optimal randomized algorithms 
 using $n$ function values compare to 
 the the error $\sigma(n+1)$ of the orthogonal projection $P_n$?
 \item If possible, find a randomized algorithm $A_n$ 
 whose error is close to $\sigma(n+1)$.
\end{itemize}
These are not new questions in the fields of Monte Carlo methods and information-based complexity.
There are several results 
for particular spaces $F$
where $e(n)$
behaves similarly to the error of $P_n$. 
See, for instance, Traub,
Wasilkowski and Woźniakowski \cite{tww},
Mathé \cite{m} and Heinrich~\cite{h}.
Results by Cohen, Davenport and Leviatan \cite{cdl}
and Cohen and Migliorati \cite{cm} contain a similar message,
see Remark~\ref{knowledge remark}.
In 1992, Novak \cite{n}
proved that
\begin{equation*}
 e(n) \geq \frac{\sigma(2n)}{\sqrt{2}}
\end{equation*}
holds for arbitrary spaces $F$.
This means that optimal randomized algorithms using $n$ function values
are never much better than the orthogonal projection $P_n$.
On the other hand, Wasilkowski and Woźniakowski \cite{ww}
proved in 2006 that
\begin{equation*}
 \sigma(n) \preccurlyeq n^{-p} (\ln n)^q \quad
 \Rightarrow \quad
 e(n) \preccurlyeq n^{-p} (\ln n)^q (\ln \ln n)^{p+1/2}
\end{equation*}
for all $p>0$ and $q\geq 0$.
Here, we write $x_n\preccurlyeq y_n$ 
if there is some $C>0$ and $n_0 \in \IN$ 
such that $x_n\leq C y_n$ for all $n \geq n_0$.
If $x_n \preccurlyeq y_n$ and $y_n \preccurlyeq x_n$,
we write $x_n\asymp y_n$.
This means that optimal randomized algorithms using function values
are always almost as good as the orthogonal projection $P_n$.
The proof of this result is constructive.
It raises the question whether
the additional power of the double logarithm is necessary or not.
In fact, Novak and Woźniakowski showed in 2012 
that this is not the case for $q=0$, that is
\begin{equation*}
 \sigma(n) \preccurlyeq n^{-p} \quad
 \Rightarrow \quad
 e(n) \preccurlyeq n^{-p}
\end{equation*}
for all $p>0$.
The proof of this result, however, is not constructive.
Both proofs can be found in their monograph \cite[Chapter~22]{track3}.
In the present paper, we prove the corresponding statement
for $q>0$.
More generally, we consider upper bounds with 
the following property.
We say that the sequence
$L:\IN\to (0,\infty)$ is \emph{regularly decreasing} 
if there is some $r\geq 0$ such that
\begin{equation}
 \label{regularly decreasing}
 L(m) \geq 2^{-r} L(n) \quad \text{whenever}\quad n \leq m\leq 2n.
\end{equation}
If there is some $n_0\in\IN$ such that $L(n)$ is nonincreasing 
for $n\geq n_0$,
this is equivalent to $L(2n) \asymp L(n)$.
Property~\eqref{regularly decreasing} is satisfied
if $L(n)n^r$ is nondecreasing.
The sequence
\begin{equation*}
 L(n) = n^{-p} \brackets{1+\log_2 n}^q
\end{equation*}
is regularly decreasing
for any $p>0$ and $q\geq 0$.
It satisfies \eqref{regularly decreasing} for $r=p$.
Another example is
\begin{equation*}
 L(n) = \brackets{1+\log_2 n}^{-q}
\end{equation*}
for any $q>0$,
which satisfies \eqref{regularly decreasing} for $r=q$.
The sequence is not regularly decreasing 
if it decays exponentially or has huge jumps.
We obtain the following result.

\begin{thm}
\label{weaker theorem}
 If $L:\IN\to (0,\infty)$ is regularly decreasing, then
 \begin{equation*}
 \sigma(n) \preccurlyeq L(n) \quad
 \Rightarrow \quad
 e(n) \preccurlyeq L(n).
\end{equation*}
\end{thm}

This solves Open Problem~99 as posed by Novak and Woźniakowski in \cite{track3}. 
One problem with this result is that it does not provide any algorithm,
it only states the existence of good algorithms.
Another problem is that the error bound is only asymptotic.
The preasymptotic behavior
of $e(n)$ may, however, be very different from its asymptotic
behavior. 
This is typically the case if the set $D$ is a domain in
high dimensional euclidean space.

These problems are tackled by Theorem~\ref{explicit theorem}.
In Section~\ref{main section}, we provide a randomized algorithm
$A_n^r$ for any $n\in\IN$ and $r\geq 0$.
This algorithm is a refinement of the algorithm proposed
by Wasilkowski and Woźniakowski \cite{ww}.
It asks for at most $n$ function values and
satisfies the following error bound.

\begin{thm}
 \label{explicit theorem}
 Assume that $L:\IN\to (0,\infty)$ satisfies \eqref{regularly decreasing}
 and let $c_r=2^{r\lceil 2r+3\rceil +1}$.
 \begin{align*}
 &\text{If}\quad
 &\sigma(n) &\leq L(n)
 \quad&\text{for all}\qquad
 &n\in\IN,\\
 &\text{then}\quad
 &e^{\rm ran}(A_n^r) &\leq c_r\, L(n)
 \quad&\text{for all}\qquad
 &n\in\IN.
 \end{align*}
\end{thm}
The constant $c_r$ only 
depends on the order $r$.
If $D$ is a domain in $d$-dimensional euclidean space,
this order is often independent of $d$ or even
strictly decreasing with $d$.
See Section~\ref{main section} for the definition of this
algorithm and several examples.

We find that the error of randomized algorithms
using $n$ function values of the target function can get very close to
the error of the orthogonal projection $P_n$
and that this is achieved by the algorithm $A_n^r$.

In Section~\ref{int section}, we use these algorithms for the integration
of functions $f$ in $F$ with respect to probability measures $\mu$.
We simply exploit the relation
\begin{equation*}
 \int_D f~\d\mu = \int_D A_n^r f~\d\mu + \int_D (f- A_n^r f)~\d\mu.
\end{equation*}
We compute the integral of $A_n^r f$
and use a direct simulation to approximate the integral of $(f\,$--$\,A_n^r f)$,
which has a small variance.
This technique is called variance reduction and widely used
for Monte Carlo integration.
See Heinrich \cite[Theorem~5.3]{h} for another example.
Even if $D$ is a high dimensional domain,
the resulting method can significantly improve on the error of a sole direct simulation
for a relatively small number of samples.

These results are based on the a priori knowledge
that our target function is contained in the unit ball
of the space $F$.
In Section~\ref{knowledge section}, we discuss
how this assumption can be weakened.

\section{The Setting}
\label{setting section}

Let $(D,\mathcal{A},\mu)$ be a measure space and $\IK\in\set{\IR,\IC}$.
The space $L^2=L^2(D,\mathcal{A},\mu)$ is the space of quadratically integrable
$\IK$-valued functions on $(D,\mathcal{A},\mu)$,
equipped with the scalar product
\begin{equation*}
 \scalar{f}{g}_2=\int_D f\cdot \overline{g}~\d\mu.
\end{equation*}
Let $F$ be a second Hilbert space and $F_\circ$ be its unit ball.
We assume that $F$ is a subset of $L^2$ and that
\begin{equation*}
 T: F \to L^2,\quad Tf=f
\end{equation*}
is compact.
With the embedding $T$ we associate a positive semi-definite and compact operator $W=T^*T$ on the space $F$.
By the spectral theorem, there is a (possibly finite) orthogonal basis $\mathcal B=\set{b_1,b_2,\dots}$ of $F$,
consisting of eigenvectors corresponding to a nonincreasing zero sequence $(\lambda_n)_{n\in\IN}$ 
of eigenvalues of $W$.
Let $N$ be the cardinality of $\mathcal{B}$.
One can easily check that $\mathcal{B}$ is orthogonal in $L^2$, as well.
We take the eigenvectors $b_n$ to be normalized in $L^2$.
We call this basis the singular value decomposition of $T$.\footnote{
This term is more commonly used to refer to the representation
$Tf=\sum_{b\in\mathcal{B}} \scalar{f}{b} Tb$ of the compact operator.
Here, the altered terminology shall ease the notation.}
The number $\sigma(n)=\sqrt{\lambda_n}$ is called its $n$-th singular value
or approximation number.

The worst case error of a deterministic algorithm $A:F\to L^2$ 
is the quantity
\begin{equation*}
 e^{\rm det}(A)= \sup\limits_{f\in F_\circ} \norm{f- A(f)}_2.
\end{equation*}
The worst case error of a measurable randomized algorithm
\begin{equation*}
 A:F\times\Omega\to L^2, \quad (f,\omega)\to A^{\omega}(f),
\end{equation*}
where $\Omega$ is the sample space of some probability space $(\Omega,\mathcal{F},\IP)$,
is the quantity
\begin{equation*}
 e^{\rm ran}(A)= \sup\limits_{f\in F_\circ} \brackets{\IE_{\omega} \norm{f- A^{\omega}(f)}_2^2}^{1/2}.
\end{equation*}
We usually skip the $\omega$ in the notation.
See Novak and Woźniakowski \cite[Chapter~4]{track1} for a precise definition of such algorithms.
We furthermore define the following minimal worst case errors
within certain classes of algorithms.

The quantity
\begin{equation*}
 e^{\rm det}(n,T,\Lambda^{\rm all}) = \inf\limits_{A\in\mathcal{A}_n^{\rm det, all}} e^{\rm det}(A)
\end{equation*}
is the minimal worst case error within the class $\mathcal{A}_n^{\rm det, all}$ of all
deterministic algorithms evaluating at most $n$ linear functionals of the input function.

The quantity
\begin{equation*}
 e^{\rm ran}(n,T,\Lambda^{\rm all}) = \inf\limits_{A\in\mathcal{A}_n^{\rm ran, all}} e^{\rm ran}(A)
\end{equation*}
is the minimal worst case error within the class 
$\mathcal{A}_n^{\rm ran, all}$ of all
measurable randomized algorithms evaluating 
at most $n$ linear functionals.

The quantity
\begin{equation*}
 e^{\rm det}(n,T,\Lambda^{\rm std}) = 
 \inf\limits_{A\in\mathcal{A}_n^{\rm det, std}} e^{\rm det}(A)
\end{equation*}
is the minimal worst case error within 
the class $\mathcal{A}_n^{\rm det, std}$ of all
deterministic algorithms evaluating at most $n$ function
values of the input function.

The quantity
\begin{equation*}
 e(n) = e^{\rm ran}(n,T,\Lambda^{\rm std}) = \inf\limits_{A\in\mathcal{A}_n^{\rm ran, std}} e^{\rm ran}(A)
\end{equation*}
finally is the minimal worst case error within the 
class $\mathcal{A}_n^{\rm ran, std}$ of all
measurable randomized algorithms evaluating 
at most $n$ function values.
This is the error to be analyzed.
It was proven by Novak \cite{n}
that
\begin{equation}
 \label{lower bound}
 e^{\rm ran}(n,T,\Lambda^{\rm std})
 \geq e^{\rm ran}(n,T,\Lambda^{\rm all})
 \geq \frac{1}{\sqrt{2}} e^{\rm det}(2n-1,T,\Lambda^{\rm all}).
\end{equation}
The error $e^{\rm det}(n,T,\Lambda^{\rm all})$
is known to coincide with $\sigma(n+1)$. 
We refer to Novak and Woźniakowski \cite[Section~4.2.3]{track1}.
The infimum is attained for the nonadaptive linear algorithm
\begin{equation*}
 P_n: F\to L^2, \quad P_n(f)=\sum_{k=1}^{n\land N} \scalar{f}{b_k}_2 b_k
.\end{equation*}
Here, $\log_2 x$ denotes the logarithm of $x>0$ in base 2,
whereas $\ln x$ denotes its natural logarithm.
The minimum of $a\in\IR$ and $b\in\IR$ is denoted by $a\land b$.
Recall that
we write
$x_n\preccurlyeq y_n$, if there is a positive constant $C$ and some $n_0\in\IN$ 
such that $x_n\leq C y_n$
for all $n \geq n_0$. 
We write $x_n\asymp y_n$ if $x_n\preccurlyeq y_n$ and $y_n\preccurlyeq x_n$.

\section{A Method for Multivariate Approximation}
\label{main section}

Let us keep the notation of the previous section.
For any $m\in\IN$ with $m\leq N$, we define
\begin{equation*}
 u_m=\frac{1}{m} \sum\limits_{j=1}^m \abs{b_j}^2.
\end{equation*}
This is a probability density 
with respect to $\mu$.
We consider
the probability measure
\begin{equation*}
 \mu_m: \mathcal{A}\to [0,1], \quad \mu_m(E)=\int_E u_m ~\d\mu
\end{equation*}
on $(D,\mathcal{A})$.
In view of optimal algorithms in $\mathcal{A}_n^{\rm det, all}$,
we introduce the following family of algorithms in $\mathcal{A}_n^{\rm ran, std}$.

\begin{alg}
Let $\boldsymbol n=\brackets{n_1,n_2,\dots}$
and $\boldsymbol m=\brackets{m_1,m_2,\dots}$ be 
sequences of nonnegative integers
such that $\boldsymbol m$ is nondecreasing
and bounded above by $N=\abs{\mathcal{B}}$. 
We define the algorithms 
$M^{(k)}_{\boldsymbol n,\boldsymbol m}: L^2\to L^2$ for $k\in\IN_0$ as follows.
\begin{itemize}
 \item Set $M^{(0)}_{\boldsymbol n,\boldsymbol m} =0$.
 \item For $k\geq 1$ and $f\in L^2$, let
 $X_1^{(k)},\hdots,X_{n_k}^{(k)}$ be random variables with distribution $\mu_{m_k}$
 that are each independent of all the other random variables and set
 \begin{equation*}
  M^{(k)}_{\boldsymbol n,\boldsymbol m} f = M^{(k-1)}_{\boldsymbol n,\boldsymbol m} f +
  \sum_{j=1}^{m_k} \left[\frac{1}{n_k}\sum_{i=1}^{n_k} \frac{\brackets{f-M^{(k-1)}_{\boldsymbol n,\boldsymbol m} f}\overline{b_j}}{u_{m_k}}
  \brackets{X_i^{(k)}} \right] b_j.
 \end{equation*}
\end{itemize}
\end{alg}

Note that the expectation of each term in the inner sum is $\langle f-M^{(k-1)}_{\boldsymbol n,\boldsymbol m} f,b_j\rangle_2$.
The algorithm $M^{(k)}_{\boldsymbol n,\boldsymbol m}$ hence approximates $f$ in $k$ steps.
In the first step, $n_1$ function values of $f$ are used for standard Monte Carlo type approximations
of its $m_1$ leading coefficients with respect to the orthonormal system $\mathcal B$.
In the second step, $n_2$ values of the residue are used for standard Monte Carlo type approximations
of its $m_2$ leading coefficients and so on.
In total, $M^{(k)}_{\boldsymbol n,\boldsymbol m}$ uses $\sum_{j=1}^k n_j$ function values of $f$.
The total number of approximated coefficients is $m_k$.

Algorithms of this type have already been studied by Wasilkowski and Woźniakowski in \cite{ww}.
The simple but crucial difference with the above algorithms
is the variable number $n_j$ of nodes in each approximation step.
Note that this stepwise approximation
is similar to several multilevel Monte Carlo methods
as introduced by Heinrich in 1998, see \cite{h multi}.

The benefit from the $k$-th step
is controlled by $m_k$ and $n_k$ as follows.

\begin{lemma}
\label{error lemma}
 For all nondecreasing sequences $\boldsymbol n$ and $\boldsymbol m$ of 
 nonnegative integers and all $k\in\IN$, we have
 \begin{equation*}
  \sigma(m_k+1)^2 \leq
  e^{\rm ran}\brackets{M^{(k)}_{\boldsymbol n,\boldsymbol m}}^2
  \leq \frac{m_k}{n_k} e^{\rm ran}\brackets{M^{(k-1)}_{\boldsymbol n,\boldsymbol m}}^2 
  + \sigma(m_k+1)^2.
 \end{equation*}
\end{lemma}

Lemma~\ref{error lemma}
corresponds to Theorem~22.14
by Novak and Woźniakowski \cite{track3}.
The setting of the present paper is
slightly more general, but the proof is the same.
Since Lemma~\ref{error lemma} is essential for the following investigation, 
I present the proof.

\begin{proof}
 The lower bound holds true, since 
 $M^{(k)}_{\boldsymbol n,\boldsymbol m}(b_{m_k+1})$ is 
 perpendicular to $b_{m_k+1}$.
 To prove the upper bound, let $f\in F_\circ$.
 By $\IE_{I}$ we denote the expectation with respect to the random variables $X_i^{(j)}$ for $j\in I$ and $i=1\hdots n_j$.
 We need to estimate
 \begin{equation*}
  \IE_{\set{1\hdots k}} \norm{f-M^{(k)}_{\boldsymbol n,\boldsymbol m}f}_2^2
  = \sum_{j=1}^N \IE_{\set{1\hdots k}} \abs{\scalar{f-M^{(k)}_{\boldsymbol n,\boldsymbol m}f}{b_j}_2}^2
 .\end{equation*}
 On the one hand, we have
 \begin{align*}
   &\sum_{j=m_k+1}^N \IE_{\set{1\hdots k}} \abs{\scalar{f-M^{(k)}_{\boldsymbol n,\boldsymbol m}f}{b_j}_2}^2
   = \sum_{j=m_k+1}^N \abs{\scalar{f}{b_j}_2}^2
   = \sum_{j=m_k+1}^N \abs{\scalar{f}{W b_j}_F}^2 \nonumber \\
   &= \sum_{j=m_k+1}^N \abs{\scalar{f}{\sigma(j) b_j}_F}^2 \sigma(j)^2
   \ \leq\ \sigma(m_k+1)^2 \norm{f}_F^2
   \ \leq\ \sigma(m_k+1)^2.
 \end{align*}
 We use the abbreviation
 \begin{equation*}
  g_j = \frac{\brackets{f-M^{(k-1)}_{\boldsymbol n,\boldsymbol m}f}\overline{b_j}}{u_{m_k}}
 \end{equation*}
 for each $j\leq m_k$. Note that $u_{m_k}=0$ implies $b_j=0$ and we set $g_j=0$ in this case.
 We then obtain on the other hand for each $j\leq m_k$ that
 \begin{equation*}
  \begin{split}
   &\IE_{\set{k}} \abs{\scalar{f-M^{(k)}_{\boldsymbol n,\boldsymbol m}f}{b_j}_2}^2
   = \IE_{\set{k}} \abs{\scalar{f-M^{(k-1)}_{\boldsymbol n,\boldsymbol m}f}{b_j}_2
   - \frac{1}{n_k}\sum_{i=1}^{n_k} g_j\brackets{X_i^{(k)}}}^2\\
   &= \IE_{\set{k}} \abs{\int_D g_j(x) ~\d\mu_{m_k}(x) - \frac{1}{n_k}\sum_{i=1}^{n_k} g_j\brackets{X_i^{(k)}}}^2\\
   &\leq \frac{1}{n_k} \int_D \abs{g_j(x)}^2 ~\d\mu_{m_k}(x)
   = \frac{1}{n_k} \int_D \abs{g_j(x)}^2 u_{m_k}(x) ~\d\mu(x)
  \end{split}
 \end{equation*}
 and hence
 \begin{equation*}
 \begin{split}
  &\sum_{j=1}^{m_k} \IE_{\set{k}} \abs{\scalar{f-M^{(k)}_{\boldsymbol n,\boldsymbol m}f}{b_j}_2}^2
  \leq \frac{1}{n_k} \int_D \sum_{j=1}^{m_k} \abs{g_j(x)}^2 u_{m_k}(x)~\d\mu(x)\\
  &= \frac{m_k}{n_k} \int_D  \abs{\brackets{f-M^{(k-1)}_{\boldsymbol n,\boldsymbol m} f}(x)}^2 ~\d\mu(x)
  = \frac{m_k}{n_k} \norm{f-M^{(k-1)}_{\boldsymbol n,\boldsymbol m} f}_2^2
 .\end{split}
 \end{equation*}
 With Fubini's theorem this yields that
 \begin{equation*}
   \IE_{\set{1\hdots k}} \norm{f-M^{(k)}_{\boldsymbol n,\boldsymbol m}f}_2^2
   \ \leq\ \frac{m_k}{n_k} \IE_{\set{1\hdots k-1}} \norm{f-M^{(k-1)}_{\boldsymbol n,\boldsymbol m} f}_2^2 
   + \sigma(m_k+1)^2
 \end{equation*}
 and the upper bound is proven.
\end{proof}

We now define the algorithm of Theorem~\ref{explicit theorem}.
We consider such algorithms $M^{(k)}_{\boldsymbol n,\boldsymbol m}$,
where the number of nodes $n_j$ is doubled in each step
and the ratio $\frac{m_j}{n_j}$ of approximated coefficients 
and computed function values is constant, say $2^{-\ell}$.
This way, the total number $m_k$ of approximated coefficients is linear in the total number $n$
of computed function values.
This is necessary to achieve an error of the same order as
with optimal algorithms using arbitrary linear information,
which precisely compute the first $n$ coefficients.
The algorithms by Wasilkowski and Woźniakowski~\cite{ww} 
do not have this property.
If the ratio is small enough, Lemma~\ref{error lemma} ensures that
$M^{(k)}_{\boldsymbol n,\boldsymbol m}$ inherits optimal error bounds
from $M^{(k-1)}_{\boldsymbol n,\boldsymbol m}$.

\begin{alg}
Given $r\geq 0$, we set $\ell_r=\lceil 2r+1 \rceil$
and define the sequences $\boldsymbol n$ and $\boldsymbol m$ by
\begin{equation*}
 n_j=\left\{\begin{array}{lr}
        0, & \text{for } j\leq \ell_r,\\
        2^{j-1}, & \text{for } j> \ell_r,
        \end{array}\right.
 \quad\quad
 m_j=\left\{\begin{array}{lr}
        0, & \text{for } j\leq \ell_r,\\
        2^{j-1-\ell_r}\land N, & \text{for } j> \ell_r.
        \end{array}\right.
\end{equation*}
For $n\in\IN$, we choose $k\in\IN_0$ such that $2^k \leq n < 2^{k+1}$
and set
\begin{equation*}
 A_n^{r} = M^{(k)}_{\boldsymbol n,\boldsymbol m}.
\end{equation*}
The algorithm $A_n^r$ obviously performs less than $n$ function evaluations.
\end{alg}

\begin{proof}[Proof of Theorem~\ref{explicit theorem}]
 Let $\boldsymbol n$ and $\boldsymbol m$ be defined as above
 and $k\in\IN_0$.
 We first show that
 \begin{equation}
 \label{almost done}
  e^{\rm ran}\brackets{M^{(k)}_{\boldsymbol n,\boldsymbol m}}
  \leq \bar c_r\, L(2^k),
 \end{equation}
 where $\bar c_r=2^{r(\ell_r+1)+1}$.
 We use induction on $k$. If $k\leq \ell_r$, we have 
 $M^{(k)}_{\boldsymbol n,\boldsymbol m}=0$ and
 \begin{equation*}
  e^{\rm ran}\brackets{M^{(k)}_{\boldsymbol n,\boldsymbol m}} 
  \,=\, \sigma(1) 
  \,\leq\, L(1) 
  \,\leq\, 2^{rk} L(2^k) 
  \,\leq\, \bar c_r\, L(2^k).
 \end{equation*}
 For $k>\ell_r$, we inductively obtain with Lemma \ref{error lemma} that
 \begin{align*}
  e^{\rm ran}\brackets{M^{(k)}_{\boldsymbol n,\boldsymbol m}}^2
  &\leq 2^{-\ell_r} e^{\rm ran}\brackets{M^{(k-1)}_{\boldsymbol n,\boldsymbol m}}^2 
  + \sigma(m_k+1)^2\\
  &\leq 2^{-\ell_r} \bar c_r^2\, L\brackets{2^{k-1}}^2 
  + L\brackets{2^{k-\ell_r-1}}^2\\
  &\leq 2^{-\ell_r}\bar c_r^2\, 2^{2r} L(2^k)^2 
  + 2^{2r(\ell_r+1)} L(2^k)^2\\
  &= \brackets{2^{2r-\ell_r} + 2^{-2}}
  \bar c_r^2\, L(2^k)^2,
 \end{align*}
 where the term in brackets is smaller than 1.
 This shows \eqref{almost done}.
 For $n\in\IN$, we choose 
 $k\in\IN_0$ with $2^k \leq n < 2^{k+1}$ and obtain
 \begin{equation*}
  e^{\rm ran}\brackets{A_n^r}
  \,=\, e^{\rm ran}\brackets{M^{(k)}_{\boldsymbol n,\boldsymbol m}}
  \,\leq\, \bar c_r\, L(2^k)
  \,\leq\, 2^r \bar c_r\, L(n)
  \,=\, c_r\, L(n),
 \end{equation*}
 as it was to be proven.
\end{proof}

Note that Theorem~\ref{weaker theorem} is a direct
consequence of Theorem~\ref{explicit theorem}.
Of course, the best possible upper bound for $\sigma(n)$ is $\sigma(n)$ itself.
If we combine Theorem~\ref{weaker theorem} for $L(n)=\sigma(n)$
with Novak's lower bound~\eqref{lower bound},
we obtain the following statement on the order of convergence.

\begin{cor}
\label{order cor}
 Assume that $\sigma(2n) \asymp \sigma(n)$. Then
 \begin{equation*}
 e^{\rm ran}\brackets{n,F\hookrightarrow L^2, \Lambda^{\rm std}}
 \asymp e^{\rm ran}\brackets{n,F\hookrightarrow L^2, \Lambda^{\rm all}}
  \asymp e^{\rm det}\brackets{n,F\hookrightarrow L^2, \Lambda^{\rm all}}.
 \end{equation*}
\end{cor}

Note that the error 
$e^{\rm det}\brackets{n,F\hookrightarrow L^2, \Lambda^{\rm std}}$
of optimal deterministic algorithms based on function values
may perform much worse, as shown by Hinrichs, Novak
and Vybíral \cite{hnv},
see also Novak and Woźniakowski \cite[Section~26.6.1]{track3}.
It is a very interesting question whether the condition on the decay
of the singular values can be relaxed.
Note that we use this condition
both to prove the upper and the lower bound of Corollary~\ref{order cor}.
On the other hand, if we combine Theorem~\ref{explicit theorem}
for $L(n)=\sigma(n)$ and the lower bound \eqref{lower bound},
we obtain the following optimality result.

\begin{cor}
 \label{optimality corollary}
 Assume that 
 there is some $r\geq 0$ such that $\sigma(2n) \geq 2^{-r} \sigma(n)$
 holds for all $n\in\IN$.
 We set $\tilde c_r=2^{r\lceil 2r+4\rceil +3/2}$.
 Then we have
 \begin{equation*}
  e^{\rm ran}\brackets{A_n^r} \leq \tilde c_r\, 
  e^{\rm ran}\brackets{n,T, \Lambda^{\rm std}}
  \qquad
  \text{for all}
  \quad
  n\in\IN.
 \end{equation*}
\end{cor}

Let us now consider some examples.
In each example, we first discuss 
the order of convergence of
$e^{\rm ran}\brackets{n,F\hookrightarrow L^2, \Lambda^{\rm std}}$.
We then talk about explicit upper bounds.

\begin{ex}[Approximation of mixed order Sobolev functions on the torus]
\label{mix example torus}
Let $D$ be the $d$-dimensional torus $\mathbb T^d$, represented by the unit cube $[0,1]^d$,
where opposite faces are identified.
Let $\mathcal{A}$ be the Borel $\sigma$-algebra on $\mathbb T^d$ and $\mu$ the Lebesgue measure.
Let $F$ be the Sobolev space
of complex valued functions on $D$ with dominating mixed smoothness $r\in\IN$,
equipped with the scalar product
\begin{equation}
\label{mix product}
 \scalar{f}{g}_F = \sum_{\norm{\alpha}_\infty \leq r} \scalar{\diff^\alpha f}{\diff^\alpha g}_2
.\end{equation}
We know that
\begin{equation*}
 e^{\rm det}\brackets{n,F\hookrightarrow L^2, \Lambda^{\rm all}} 
 \asymp n^{-r}\ln^{r(d-1)} n.
\end{equation*}
This classical result goes back to Babenko \cite{babenko} and Mityagin \cite{mityagin}.
Corollary~\ref{order cor} yields
\begin{equation*}
 e^{\rm ran}\brackets{n,F\hookrightarrow L^2, \Lambda^{\rm std}} 
 \asymp n^{-r}\ln^{r(d-1)} n.
\end{equation*}
This is a new result.
The optimal order is achieved by the algorithm $A_n^r$
and the author does not know of any other algorithm with this property.
It is still an open problem whether the same rate can be achieved
with deterministic algorithms based on function values.
So far, it is only known that
\begin{equation*}
 n^{-r}\ln^{r(d-1)}n
 \preccurlyeq
 e^{\rm det}\brackets{n,F\hookrightarrow L^2, \Lambda^{\rm std}}
 \preccurlyeq
 n^{-r}\ln^{(r+1/2)(d-1)} n.
\end{equation*}
The upper bound is achieved by
Smolyak's algorithm, see Sickel and Ullrich~\cite{su}.

We now turn to explicit estimates.
We know that there is some $C_{r,d}>0$ such that
\begin{equation}
\label{asymptotic bound mix}
 e^{\rm ran}\brackets{n,F\hookrightarrow L^2, \Lambda^{\rm std}} 
 \leq C_{r,d}\, n^{-r}\ln^{r(d-1)} n
 \quad\text{for all}\quad n\geq 2.
\end{equation}
This upper bound is optimal as $n$ tends to infinity.
However, it is not useful to describe the error
numbers for small values of $n$.
Simple calculus shows that the right hand side in
\eqref{asymptotic bound mix} is increasing for $n\leq e^{d-1}$.
The error numbers, on the other hand, are decreasing.
Moreover,
the right hand side attains its minimum for $n=2$
if restricted to $n\leq (d-1)^{d-1}$
and is hence larger than $e^{\rm ran}\brackets{2,F\hookrightarrow L^2, \Lambda^{\rm std}}$.
This means that the trivial upper bound
\begin{equation*}
 e^{\rm ran}\brackets{n,F\hookrightarrow L^2, \Lambda^{\rm std}}
 \leq 
 e^{\rm ran}\brackets{2,F\hookrightarrow L^2, \Lambda^{\rm std}}
 \quad\text{for all}\quad n\geq 2
\end{equation*}
is better than \eqref{asymptotic bound mix} for all $n\leq (d-1)^{d-1}$
and regardless of the value of $C_{r,d}$.
For these reasons, it is important 
to consider different error bounds,
if the dimension $d$ is large.
See also the paper of Kühn, Sickel and Ullrich \cite{ksu}. 
Based on this paper,
it is shown by the author \cite{ich16} that
\begin{equation*}
 \sigma(n) \leq \brackets{2/n}^p 
 \quad\text{for all}\quad 
 n\in\IN,
 \quad\text{if}\quad
 p=\frac{r}{2+\ln d}.
\end{equation*}
We obtain with Theorem \ref{explicit theorem} that
\begin{equation}
\label{preasymptotic bound mix}
 e^{\rm ran}\brackets{A^p_n}
  \leq 2\cdot \brackets{2^{\lceil 2p+4\rceil}/n}^p 
  \quad\text{for}\quad n\in\IN.
\end{equation}
\end{ex}

\begin{ex}[Approximation of mixed order Sobolev functions on the cube]
\label{mix example cube}
Now, let $D$ be the $d$-dimensional unit cube $[0,1]^d$ with the induced topology and let
$\mathcal{A}$ be the Borel $\sigma$-algebra and $\mu$ the Lebesgue measure.
Let $F$ be the Sobolev space
of complex valued functions on $[0,1]^d$ with dominating mixed smoothness $r\in\IN$,
equipped with the scalar product \eqref{mix product}.
Just like on the torus, 
we have
\begin{equation*}
 e^{\rm ran}\brackets{n,F\hookrightarrow L^2, \Lambda^{\rm std}} 
 \asymp e^{\rm det}\brackets{n,F\hookrightarrow L^2, \Lambda^{\rm all}}
 \asymp n^{-r}\ln^{r(d-1)}n,
\end{equation*}
where the optimal rate is achieved by $A_n^r$.
Like in Example~\ref{mix example torus}, 
the corresponding upper bounds are bad for $n\leq (d-1)^{d-1}$.
In this range, we need different estimates for the approximation
numbers. It is known that
\begin{equation*}
 \sigma(n) \leq (2/n)^{p} \quad\text{for}\quad
 n\in\IN, \quad\text{if}\quad
 p=\frac{1.1929}{2+\ln d}.
\end{equation*}
This estimate cannot be improved significantly for $n\leq 2^d$, even if $r=\infty$.
See the author's paper~\cite{ich16} for more details.
With Theorem~\ref{explicit theorem}, we obtain the upper bound
\begin{equation*}
 e^{\rm ran}\brackets{A^p_n}
  \leq 2\cdot (2^6/n)^{p} \quad\text{for}\quad n\in\IN.
\end{equation*}
\end{ex}

\begin{ex}[Approximation in tensor product spaces]
 This example is more general than the previous ones.
 By $H_1\otimes H_2$ we denote the tensor product of 
 two Hilbert spaces $H_1$ and $H_2$.
 For $j=1\hdots d$ let $(D_j,\mathcal{A}_j,\nu_j)$ be a $\sigma$-finite
 measure space and $F_j$ be a Hilbert space of $\IK$-valued functions
 which is compactly embedded in $L^2(D_j,\mathcal{A}_j,\nu_j)$.
 The $\sigma$-finity of the measure spaces ensures that
 \begin{equation*}
  L^2(D_1,\mathcal{A}_1,\nu_1) \otimes \dots \otimes L^2(D_d,\mathcal{A}_d,\nu_d) = L^2(D,\mathcal{A},\mu),
 \end{equation*}
 where $D$ is the Cartesian product of the sets $D_j$
 and $\mu$ is the unique product measure of the measures $\nu_j$
 on the tensor product $\mathcal{A}$ of the $\sigma$-algebras $\mathcal{A}_j$.
 The tensor product space
 \begin{equation*}
  F=F_1 \otimes \dots\otimes F_d
 \end{equation*}
 is compactly embedded in $L^2(D,\mathcal{A},\mu)$.
 Assuming that the approximation numbers
 of the univariate embeddings $F_j\hookrightarrow L^2(D_j,\mathcal{A}_j,\nu_j)$ are of polynomial decay, that is
 \begin{equation*}
  e^{\rm det}\brackets{n,F_j\hookrightarrow L^2(D_j,\mathcal{A}_j,\nu_j), \Lambda^{\rm all}}
  \asymp n^{-r_j}
 \end{equation*}
 for some $r_j>0$, it can be derived from
 Mityagin \cite{mityagin} and Nikol'skaya \cite{nikolskaya} that
 \begin{equation*}
  e^{\rm det}\brackets{n,F\hookrightarrow L^2(D,\mathcal{A},\mu), \Lambda^{\rm all}}
  \asymp n^{-r}\ln^{r(d_0-1)}n,
 \end{equation*}
 where $r$ is the minimum among all numbers $r_j$ and $d_0$ is its multiplicity.
 Corollary~\ref{order cor} implies
\begin{equation*}
 e^{\rm ran}\brackets{n,F\hookrightarrow L^2(D,\mathcal{A},\mu),
 \Lambda^{\rm std}}
  \asymp n^{-r}\ln^{r(d_0-1)}n,
\end{equation*}
where the optimal order is achieved by $A_n^r$.
We do not discuss explicit estimates in this abstract setting.
\end{ex}

\begin{ex}[Approximation of isotropic Sobolev functions on the torus]
\label{isotropic example}
Let $D$ again be the $d$-torus,
this time represented by $[0,2\pi]^d$.
Let $F$ be the Sobolev space
of complex valued functions on $D$ with isotropic smoothness $r\in\IN$,
equipped with the scalar product
\begin{equation*}
 \scalar{f}{g}_F = \sum_{\norm{\alpha}_1 \leq r} \scalar{\diff^\alpha f}{\diff^\alpha g}_2
.\end{equation*}
This example is not a tensor product problem.
For this classical problem, it is known that
\begin{multline*}
e^{\rm det}\brackets{n,F\hookrightarrow L^2, \Lambda^{\rm std}}
\asymp
e^{\rm ran}\brackets{n,F\hookrightarrow L^2, \Lambda^{\rm std}}\\
\asymp
e^{\rm det}\brackets{n,F\hookrightarrow L^2, \Lambda^{\rm all}}
\asymp
e^{\rm ran}\brackets{n,F\hookrightarrow L^2, \Lambda^{\rm all}}
\asymp
n^{-r/d}
\end{multline*}
for $r>d/2$. In the case $r\leq d/2$, where function values
are only defined almost everywhere, the last three relations stay valid. 
See Jerome \cite{j67}, Triebel \cite{t}, Mathé \cite{m} and Heinrich \cite{h iso}.
For $n\leq 2^d$, however, the function $n^{-r/d}$ is not suited to
describe the behavior of $\sigma(n)$.
It has been proven by Kühn, Mayer and Ullrich \cite{kmu}
that there are positive constants $b_r$ and $B_r$ that do not
depend on $d$ such that
\begin{equation}
\label{counting}
 b_r \brackets{\frac{\log_2\brackets{1+d/\log_2 n}}{\log_2 n}}^{r/2}
 \leq \sigma(n) \leq 
 B_r \brackets{\frac{\log_2\brackets{1+d/\log_2 n}}{\log_2 n}}^{r/2}
\end{equation}
for all $d>1$ and $n\in\IN$ with $d\leq n \leq 2^d$.
If we apply Relation~\eqref{lower bound}
and Theorem~\ref{explicit theorem}\footnote{We take $L(n)$
as the right hand side in \eqref{counting} for $d\leq n \leq 2^d$,
$L(n)=L(2^d)$ for $n>2^d$ and $L(n)=\max\set{1,L(d)}$
for $n<d$.
Then $\sigma(n)\leq L(n)$ for $n\in\IN$ and $L(n)n^r$ is nondecreasing.}, 
we obtain the existence of $d$-independent positive constants $\tilde b_r$
and $\widetilde B_r$ such that
\begin{equation*}
 \tilde b_r \brackets{\frac{\log_2\brackets{1+d/\log_2 n}}{\log_2 n}}^{r/2}
 \leq e(n) 
 \leq 
 \widetilde B_r \brackets{\frac{\log_2\brackets{1+d/\log_2 n}}{\log_2 n}}^{r/2}
\end{equation*}
for all $d>1$ and $n\in\IN$ with $d\leq n \leq 2^{d-1}$.
This optimal behavior is achieved by the algorithm $A_n^r$.
\end{ex}

\begin{rem}[Implementation of these algorithms]
\label{implementation approx}
The construction of the algorithms $A_n^r$ is completely explicit.
We are able to implement these algorithms,
if we know the singular value decomposition $\mathcal B$ of the embedding $F\hookrightarrow L^2$
and if we are able to sample from the probability distributions $\mu_m$.
This task may be very hard.
In Example~\ref{mix example torus} and \ref{isotropic example}, however, it is not.
Here, $\mathcal B$ is the Fourier basis of $L^2$
and all the random variables are independent and uniformly distributed on the unit cube.
Also the case of general tensor product spaces $F$ and $L^2$ can be handled,
if the singular value decompositions $\mathcal{B}_j$
of the univariate embeddings $F_j\hookrightarrow L^2(D_j,\mathcal{A}_j,\nu_j)$ are known.
Then, the singular value decomposition of the embedding $F\hookrightarrow L^2$ is given by
\begin{equation*}
 \mathcal{B}=\set{b^{(1)}\otimes\dots\otimes b^{(d)} \mid b^{(j)}\in\mathcal{B}_j \text{ for }j=1\dots d}
\end{equation*}
and the probability measure $\mu_m$ is the average of $m$ product densities, that is
\begin{equation*}
 \mu_m=\frac{1}{m} \sum_{i=1}^m \bigotimes_{j=1}^d \eta_{i,j},
\end{equation*}
where $\d\eta_{i,j}=|b_{i,j}|^2\d\nu_j$ with some $b_{i,j}\in\mathcal{B}_j$.
A random sample $x$ from this distribution can be obtained as follows:
\begin{itemize}
 \item[(1)] Get $i$ from the uniform distribution on $\set{1,\dots,m}$.
 \item[(2)] Get $x_1,\dots,x_d$ independently from the probability distributions $\eta_{i,1},\dots,\eta_{i,d}$.
\end{itemize}
The second step can for example be done by rejection sampling,
if the measures $\eta_{i,j}$ have a bounded Lebesgue density.
This way, the total sampling costs are linear in $d$.
Another method of sampling from $\mu_m$ is proposed
by Cohen and Migliorati in \cite[Section 5]{cm}.
\end{rem}

\section{A Method for Multivariate Integration}
\label{int section}

In this section, we require the measure $\mu$ to be finite.
This ensures that the integral operator
\begin{equation*}
 I: F \to \IK, \quad I(f)=\int_D f d\mu
\end{equation*}
is well defined and continuous on $F$.
Let us assume that $\mu$ is a probability measure.
We want to approximate $I(f)$ for an unknown function $f\in F_\circ$
by a randomized algorithm $Q_n$ which evaluates at most $n$ function values
of $f$.
The worst case error of $Q_n$ is the quantity
\begin{equation*}
 e^{\rm ran}(Q_n)=\sup\limits_{f\in F_\circ} \brackets{\IE\abs{I(f)-Q_n(f)}^2}^{1/2}
.\end{equation*}
The minimal worst case error among such algorithms is denoted by
\begin{equation*}
 e^{\rm ran}(n,I,\Lambda^{\rm std}) = \inf\limits_{Q_n} e^{\rm ran}(Q_n)
.\end{equation*}
Like any method for $L^2$-approximation,
the algorithm $A_n^r$ from Section $\ref{main section}$
can also be used for numerical integration.

\begin{alg}
 For all $r>0$, any $n\in\IN$ and $f\in L^2$, let
 \begin{equation*}
  Q_{2n}^r(f)=I(A_n^r f) + \frac{1}{n}\sum_{j=1}^n \brackets{f-A_n^r f}(X_j),
 \end{equation*}
 where $X_1,\dots,X_n$ are random variables with distribution $\mu$ which are
 independent of each other and the random variables in $A_n^r$.
\end{alg}

It is easy to verify that $Q_{2n}^r$ is unbiased, evaluates at most 
$2n$ function values of $f$
and satisfies
\begin{equation*}
  \IE\abs{I(f)-Q_{2n}^r(f)}^2
  \leq \frac{1}{n}\, \IE \norm{f-A_n^r f}_2^2
\end{equation*}
for each $f$ in $L^2$. 
We thus obtain the following corollary.

\begin{cor}
\label{integration cor}
 Assume that 
 $L:\IN\to (0,\infty)$ satisfies \eqref{regularly decreasing}
 and let $c_r=2^{r\lceil 2r+3\rceil +1}$.
 \begin{align*}
 &\text{If}\quad
 &\sigma(n) &\leq L(n)
 \quad&\text{for all}\qquad
 &n\in\IN,\\
 &\text{then}\quad
 &e^{\rm ran}(Q_{2n}^r) &\leq c_r\,n^{-1/2} L(n)
 \quad&\text{for all}\qquad
 &n\in\IN.
 \end{align*}
 In particular:
 \begin{align*}
 e^{\rm det}\brackets{n,F\hookrightarrow L^2, \Lambda^{\rm all}}
  &\preccurlyeq n^{-p} \ln^q n\\
 \Rightarrow\quad
 & e^{\rm ran}\brackets{n,I, \Lambda^{\rm std}}
  \preccurlyeq n^{-p-1/2} \ln^q n.
 \end{align*}
\end{cor}

The result on the order of convergence is quite general but not always optimal.
An example is given by integration with respect to the Lebesgue measure $\mu$
on the Sobolev space $F$ with dominating mixed smoothness $r$
on the $d$-dimensional unit cube, 
as treated by Novak and the author \cite{kn} and
Ullrich \cite{ullrich}. In this case, we have
\begin{align*}
 e^{\rm det}\brackets{n,F\hookrightarrow L^2, \Lambda^{\rm all}}
  &\asymp n^{-r} \ln^{r(d-1)} n,\\
 e^{\rm ran}\brackets{n,I, \Lambda^{\rm std}}
  &\asymp n^{-r-1/2}.
\end{align*}
The main strength of Corollary \ref{integration cor} is
that it provides us with unbiased methods for high dimensional integration
achieving a small error with a modest number of function values.

\begin{ex}[Integration of mixed order Sobolev functions on the torus]
\label{mix example torus integration}
Like in Example~\ref{mix example torus},
let $F$ be the Sobolev space of dominating mixed smoothness $r$ on the $d$-torus
and let $\mu$ be the Lebesgue measure.
Among all randomized algorithms for multivariate integration in $F$
the randomized Frolov algorithm $Q_n^*$ is known to have the optimal error rate.
It is shown by Ullrich \cite{ullrich} that there is some constant $c>2^d$ such that
\begin{equation}
\label{Frolov bound}
 e^{\rm ran}\brackets{Q^*_n}
  \leq c\,n^{-r-1/2} \quad\quad\text{for } n\in\IN.
\end{equation}
However, this estimate is trivial,
if $n$ is not exponentially large in $d$.
For smaller values of $n$, an error less than one is guaranteed 
by the direct simulation 
\begin{equation*}
 S_n(f)=\frac{1}{n} \sum_{j=1}^n f(X_j),
\end{equation*}
with independent and uniformly distributed random variables $X_j$.
It satisfies
\begin{equation}
\label{Standard bound}
 e^{\rm ran}\brackets{S_n}
  \leq n^{-1/2} \quad\quad\text{for } n\in\IN.
\end{equation}
However, this error bound converges only slowly, as $n$ tends to infinity.
It does not reflect the smoothness of the integrands at all.
The above method 
also guarantees nontrivial error bounds for smaller values of $n$,
but converges faster than $S_n$.
Relation~(\ref{preasymptotic bound mix}) immediately yields
that
\begin{equation}
\label{preasymptotic bound mix integration}
 e^{\rm ran}\brackets{Q^p_{2n}}
  \leq C\,n^{-p-1/2} \quad\quad\text{for } n\in\IN
\end{equation}
with $p=\frac{r}{2+\ln d}$ and $C=2^{p \lceil 2p+4\rceil +1}$.
For example, let $d=500$ and $r=8$.
For one million function values, the estimate~(\ref{Frolov bound})
for the Frolov algorithm is larger than one,
the estimate~(\ref{Standard bound}) for the direct simulation gives the error $10^{-3}$
and the estimate~(\ref{preasymptotic bound mix integration}) for our new algorithm
gives an error smaller than $5\cdot10^{-7}$.
\end{ex}

\begin{rem}[Implementation of these algorithms]
We are able to implement the algorithms $Q_{2n}^r$
under the following assumptions:
\begin{itemize}
 \item We are able to implement $A_n^r$.
 This issue is discussed in Remark~\ref{implementation approx}.
 \item We know the integrals $I(b_j)$ of the eigenfunctions $b_j\in\mathcal{B}$
 for all $j\leq 2^{-\ell_r}n$.
 \item We can sample from the probability distribution $\mu$.
\end{itemize}
In the above example, the implementation is particularly easy,
since $\mathcal{B}$ is the Fourier basis and
all the random variables are independent and uniformly distributed on the unit cube.
\end{rem}

\section{A weaker type of a priori knowledge}
\label{knowledge section}

In the previous sections,
we assumed that the target function $f$
is contained in the unit ball of a Hilbert space $F$
which is compactly embedded into $L^2$,
that is
\begin{equation}
 \label{stronger assumption}
 \norm{f}_F
  \leq 1.
\end{equation}
As we have seen in Section~\ref{setting section},
the space $F$ induces a nonincreasing sequence $\sigma$,
the singular numbers
\begin{equation*}
  \sigma(1) \geq \sigma(2) \geq \hdots > 0
\end{equation*}
of the embedding $F\hookrightarrow L^2$.
This sequence is either finite or tends to zero.
It also induces a nested sequence $V$ of subspaces
\begin{equation*}
  V_0 \subset V_1 \subset V_2 \subset \hdots \subset L^2,
  \qquad
  \dim(V_m) = m,
\end{equation*}
 where $V_m$ is spanned by the first $m$ elements
 of the singular value decomposition.
 
 In turn, any such pair $\brackets{\sigma,V}$
 induces a Hilbert space $F$ which is compactly embedded into $L^2$.
 We choose $b_m$ as an element of the orthogonal complement 
 of $V_{m-1}$ in $V_{m}$ with $\norm{b_m}_2=1$ 
 and define $F$ by its orthonormal basis 
 $\{\sigma(1) b_1,\sigma(2) b_2,\hdots\}$. 
 It has the scalar product
 \begin{equation*}
 \scalar{f}{g}_F =
  \sum \sigma(k)^{-2} \scalar{f}{b_k}_2 \overline{\scalar{g}{b_k}}_2,
 \end{equation*}
 where we take the sum over the whole sequence $\sigma$.
 It is not hard to see that the correspondence between $F$ 
 and the pair $\brackets{\sigma,V}$ is bijective
 up to the choice of the spaces $V_m$ 
 for which we have $\sigma(m+1)=\sigma(m)$.
 
 Let $P_m$ denote the orthogonal projection onto $V_m$ in $L^2$.
 It is readily verified that our assumption 
 \eqref{stronger assumption} on the target function $f$ implies that
 \begin{equation}
 \label{weak assumption}
  \norm{f-P_m f}_2^2 \leq \sigma(m+1)^2
  \quad\text{for all}
  \quad
  m\in\IN_0.
 \end{equation}
 In general, however, \eqref{weak assumption} is strictly weaker 
 than \eqref{stronger assumption}.
 For example, if $\sigma(k)=1/k$ for $k\in\IN$, the function 
 \begin{equation*}
  f=\sum (\sigma(k)^2-\sigma(k+1)^2)^{1/2}\, b_k
 \end{equation*}
 satisfies \eqref{weak assumption}
 but is not even contained in the space $F$.
 In Section~\ref{main section},
 we constructed a randomized algorithm $A_n^r: L^2 \to V_m$
 and proved upper bounds on the mean square 
 error $\IE \norm{f- A_n^r(f)}_2^2$
 for any $f$ from \eqref{stronger assumption}.
 In fact, the same error bounds hold for any $f$ from \eqref{weak assumption}.
 We state this as Theorem~\ref{subspaces thm}.
 
 \begin{thm}
  \label{subspaces thm}
  Let $(D,\mathcal{A},\mu)$ be a measure space 
  and $L^2=L^2(D,\mathcal{A},\mu)$. 
  For any $m\in\IN_0$ let $V_m$ be an $m$-dimensional
  subspace of $L^2$ such that $V_m\subset V_{m+1}$ 
  and let $P_m: L^2\to V_m$ be the orthogonal projection onto $V_m$.
 Assume that $f\in L^2$ satisfies
 \begin{equation}
 \label{approximability}
  \norm{f-P_m f}_2^2 \leq \varepsilon(m)
  \qquad\text{for all}\quad m\in \IN_0 
 \end{equation}
 with some $\varepsilon:\IN_0\to (0,\infty)$.
 Then the randomized algorithm
 $Q_m:L^2 \to V_m$
 as defined below satisfies
 \begin{equation*}
  \IE \norm{f- Q_m f}_2^2 \leq 2\, \varepsilon(m)
 \end{equation*}
 for any $m=2^k$ and $k\in\IN_0$.
 The number of requested function values is at most
 \begin{equation}
 \label{costs}
  n\brackets{Q_m} 
  \leq 4\, m \cdot \max\limits_{0\leq j \leq k} 
  \left\lceil\frac{\varepsilon(\lfloor 2^{j-1}\rfloor)}{\varepsilon(2^j)}\right\rceil.
 \end{equation}
 \end{thm}

To define the algorithm $Q_m$ we
choose $b_n$ in the orthogonal complement
of $V_{n-1}$ in $V_n$ with $\norm{b_n}_2=1$
for all $n\in\IN$.
For $j\in\IN$, we set
\begin{equation*}
 m_j=2^{j-1}
 \qquad\text{and}\qquad
 n_j=2^j \left\lceil 
 \frac{\varepsilon(\lfloor 2^{j-2}\rfloor)}{\varepsilon(2^{j-1})}
 \right\rceil.
\end{equation*}
Then the method $M^{(k)}_{\boldsymbol n,\boldsymbol m}: L^2\to V_{m_k}$
for $k\in\IN_0$ can be defined as in Section~\ref{main section}.
Given $m=2^k$ for some $k\in\IN_0$,
we define $Q_m=M^{(k+1)}_{\boldsymbol n,\boldsymbol m}: L^2\to V_m$.

\begin{proof}
 We only sketch the proof since it is very similar to
 the proof of Theorem~\ref{explicit theorem}.
 Just like in Lemma~\ref{error lemma}, we can show 
 for any $k\in\IN_0$ that
 \begin{equation*}
  \IE \norm{f- M^{(k+1)}_{\boldsymbol n,\boldsymbol m} f}_2^2
  \leq 
  \frac{m_{k+1}}{n_{k+1}}\, \IE \norm{f- M^{(k)}_{\boldsymbol n,\boldsymbol m} f}_2^2
  + \varepsilon(m_{k+1}).
 \end{equation*}
 The statement follows by induction on $k\in \IN_0$.
\end{proof}

Note that we did not impose any condition on 
the upper bound $\varepsilon:\IN_0\to (0,\infty)$.
If $\varepsilon$ is regularly decreasing, the maximum in \eqref{costs}
is bounded by a constant which does not depend on $m$.
Roughly speaking, the algorithm $Q_m$ admits a
mean square error of order $\varepsilon(m)$
with a sample size of order $m$
for any $f$ from \eqref{approximability}.

\begin{rem}[Optimal approximation within a subspace]
\label{knowledge remark}
Let $D$ be a Borel subset of $\IR^d$
with positive Lebesgue measure,
$\mathcal{A}$ be the Borel sigma algebra on $D$
and $\mu$ be a probability measure on $(D,\mathcal{A})$.
The best approximation of $f\in L^2(D,\mathcal{A},\mu)$ within the
subspace $V_m$ is given by $P_m f$.
Its error is given by the number
\begin{equation*}
 e_m(f)=\inf\limits_{v\in V_m} \norm{f-v}_2 = \norm{f-P_m f}_2.
\end{equation*}
In general, we cannot find $P_m f$
by sampling only a finite number 
of function values of $f$.
What we can provide, is a random
approximation $v_m$ within $V_m$ 
whose root mean square error 
\begin{equation*}
 \brackets{\IE \norm{f-v_m}_2^2}^{1/2}
\end{equation*}
is close to $e_m(f)$.
If we know the numbers $e_m(f)$ for all $m\in\IN$
(or some good upper bound)
and if they are regularly decreasing,
we can choose $v_m$ as the output of the method $Q_m$
from Theorem~\ref{subspaces thm},
which uses a sample size of order $m$.
But even if we do not know anything about $f\in L^2$,
we can still find an approximation $v_m$ like above.
We only need the mild assumption that $V_m$ consists of
functions defined everywhere on $D$ and that for each $x\in D$, there is 
some $v\in V_m$ with $v(x)\neq 0$.
We can then choose $v_m$ as the output of a 
weighted least squares method,
see Cohen and Migliorati \cite[Theorem~2.1~(iv)]{cm}.
The sample size of this method, however, is
at least of order $m\ln m$.
In both cases, the involved proportionality constants
are independent of the dimension of the domain $D$.
\end{rem}

\begin{ack}
 I wish to thank Erich Novak, Robert Kunsch, Winfried Sickel and two
 anonymous referees, whose comments and questions
 led to the present generality of the theorems.
\end{ack}

\raggedright{

}

\end{document}